\documentclass[reqno,12pt]{article}
\usepackage{amsmath}
\usepackage{amsthm}
\usepackage{amssymb}
\usepackage{enumerate}
\usepackage{authblk}
\makeatletter
\renewcommand{\p@enumii}{} 
\renewcommand{\p@enumiii}{}
\makeatother
\newcommand{\subs}{\subseteq}

\providecommand{\definitionname}{Definition}
\newtheorem{theorem}{Theorem} [section]
\newtheorem{lemma}[theorem]{Lemma}

\newtheorem{fact}[theorem]{Fact}

\newtheorem{cor}[theorem]{Corollary}
\newtheorem{conj}{Conjecture}

\def\N{\mathbb N}
\def\R{\mathbb R}
\def\RR{\tilde R}  
\def\rr{\tilde{r}}
\def\sr{\hat{r}}
\def\Codd{{\mathcal C}_{\text odd}}
\def\Con{\mathcal Con}
\def\cC{\mathcal C}
\def\cG{\mathcal G}
\def\cH{\mathcal H}

\begin{document}

\title{Online size Ramsey numbers: Odd cycles vs connected graphs}
\author{Grzegorz Adamski}
\author{Ma\l gorzata Bednarska-Bzd\c ega}
\affil{Faculty of Mathematics and CS, Adam Mickiewicz University in Pozna\'n}

\maketitle

\begin{abstract}
Given two graph families $\cH_1$ and $\cH_2$, a size Ramsey game is  
played on the edge set of $K_\N$. In every round, Builder selects an edge and Painter colours 
it red or blue.  Builder is trying to force Painter to create a red copy of a graph from $\cH_1$ or a blue copy of a graph from
$\cH_2$ as soon as possible. The online (size) Ramsey number $\rr(\cH_1,\cH_2)$ is the smallest number of rounds in the game
provided  Builder and Painter play optimally. 
We prove that if $\cH_1$ is the family of all odd cycles and $\cH_2$ is the family of all connected graphs on $n$ vertices
and $m$ edges, then $\rr(\cH_1,\cH_2)\ge  \varphi n + m-2\varphi+1$, where $\varphi$ is the golden ratio, and for 
$n\ge 3$, $m\le (n-1)^2/4$ we have $\rr(\cH_1,\cH_2)\le n+2m+O(\sqrt{m-n+1})$. We also show that $\rr(C_3,P_n)\le 3n-4$ for $n\ge 3$.
As a consequence we get $2.6n-3\le \rr(C_3,P_n)\le 3n-4$ for every $n\ge 3$.
\end{abstract}

\section{Introduction}

Let $\cH_1$ and $\cH_2$ be nonempty families of finite graphs. Consider the following game  $\RR(\cH_1, \cH_2)$
between Builder and Painter, played on the infinite board $K_\N$ (i.e. the board is a complete graph with the vertex set $\N$). In every round, Builder chooses a previously unselected
edge of $K_\N$ and Painter colours it red or blue. The game ends if after a move of Painter there is a red copy of a graph $\cH_1$ 
or a blue copy of a graph from $\cH_2$.  Builder tries to finish the game as soon as possible, while the goal of Painter is the opposite.
Let $\rr(\cH_1, \cH_2)$ be the minimum number of rounds in the game $\RR(H_1, H_2)$, provided both players play optimally. 
If $\cH_i$ consists of one graph $H_i$, we simply write  $\rr(H_1, H_2)$ and $\RR(H_1, H_2)$ and call them an online size Ramsey 
number and an online size Ramsey game, respectively. 
In the literature, online size Ramsey numbers are called also online Ramsey numbers, 
which can be a bit confusing. The online size Ramsey number  $\rr(H_1, H_2)$ is a game-theoretic counterpart of the classic size Ramsey number 
$\sr(H_1, H_2)$, i.e.~the minimum number of edges in a graph with the property that every two-colouring of its edges results in
a red copy of $H_1$ or a blue copy of $H_2$.  One of the most interesting questions in online size Ramsey theory is related to the   
Ramsey clique game $\RR(K_n, K_n)$,  introduced by Beck \cite{beck}. The following problem, attributed by Kurek and Ruci\'nski 
\cite{ar} to V.~R\H odl, is still open:

\begin{conj}
$\frac{\rr(K_n,K_n)}{\sr(K_n,K_n)}\to 0\text{ for } n\to\infty.$
\end{conj}

Conlon \cite{con} proved the above is true for an infinite increasing sequence of cliques. It is natural to ask for other natural graph 
sequences $H_n$ with their online size Ramsey numbers much smaller than their 
size Ramsey numbers, i.e. $\rr(H_n,H_n)=o(\sr(H_n,H_n))$. 
Bla\v zej, Dvo\v r\' ak and Valla \cite{ind} constructed an infinite sequence of increasing trees $T_n$ such that $\rr(T_n,T_n)=o(\sr(T_n,T_n))$.
On the other hand, Grytczuk, Kierstead and Pra\l at \cite{gryt} constructed a family of trees $B_n$ on $n$ vertices such that both
$\rr(B_n,B_n)$ and $\sr(B_n,B_n)$ are of order $n^2$. 

Suppose that for a sequence of graphs $H_n$ on $n$ vertices we know that $\sr(H_n,H_n)$ (or $\sr(H_n,G)$ with a fixed graph $G$) is linear. 
Then clearly also $\rr(H_n,H_n)$ (or $\rr(H_n,G)$) is linear. 
In such cases, it is usually not easy to find multiplicative constants in the lower and the upper bound respectively, which do not differ much. 
For example, for paths $P_n$ on $n$ vertices we have 
$2n-3\le \rr(P_n,P_n)\le 4n-7$, for cycles $C_n$  we have  $2n-1\le \rr(C_n,C_n)\le 72n-3$ 
(in the case of even cycles the multiplicative constant in the upper bound can be improved to $71/2$). 
The lower bound in both examples comes from an observation that in general $\rr(H_1, H_2)\ge |E(H_1)|+|E(H_2)|-1$, which is a result
of the following strategy of Painter: she colours every selected edge $e$ red unless $e$ added to a red subgraph 
present on the board would create a red copy of $H_2$.   The upper bound for paths was proved in \cite{gryt}, while the upper bound for
cycles -- in \cite{ind}. 
The exact results on $\rr(H_1,H_2)$ are rare, except for very small graphs $H_1$, $H_2$.
It is known that $\rr(P_3,P_n)=\lceil 5(n-1)/4\rceil$ and $\rr(P_3,C_n)=\lceil 5n/4\rceil$ \cite{lo}.

The studies on games $\RR(\cH_1, \cH_2)$ for graph class  $\cH_1$, $\cH_2$ such that at least one class is infinite, are not so
popular as their one-graph versions.
Nonetheless, they appear implicitly in the analysis of one-graph games $\RR(H_1, H_2)$ since a standard method of getting a lower bound on 
$\rr(H_1,H_2)$ is to assume that Painter  colours every selected edge $e$ red unless $e$ added to a red subgraph 
present on the board would create a red copy of a graph from $\cH$, for some graph class $\cH$. 
For example Cyman, Dzido, Lapinskas and Lo \cite{lo} analysed such a strategy  while studying the game $\RR(C_k,H)$ for a connected graph $H$
and implicitly obtained that 
\begin{equation}\label{cyccon}
\rr(\cC,\Con_{n,m})\ge n+m-1,
\end{equation} 
where $\cC$ denotes the class of all cycles
and $\Con_{n,m}$ is the class of all connected graphs with exactly $n$ vertices and at least $m$ edges. 

Infinite graph classes are also studied explicitly in the size Ramsey theory and in its online counterpart. 
Dudek, Khoeini and Pra\l at \cite{dud} initiated the study on the size Ramsey number $\sr(\cC, P_n)$. 
Bal  and Schudrich \cite{bal} proved that $2.06n-O(1)\le \sr(\cC,P_n)\le 21/4 n+27$.
The online version of this problem was considered by  Schudrich \cite{ms}, who showed that $\rr(\cC, P_n)\le 2.5 n+5$. 
Here the multiplicative constant is quite close to the constant 2 in the lower bound, given by \eqref{cyccon}. 

Let $\Codd$ denote the class of all odd cycles and $S_n$ be a star with $n$ edges (i.e.~$S_n=K_{1,n}$). 
Pikurko \cite{oleg}, while disproving Erd\H os' conjecture that $\sr(C_3, S_n)=(3/2+o(1))n^2$, proved that
$\sr(\Codd, S_n)=(1+o(1))n^2$.  As for the online version of this problem, it is not hard to see that 
Builder can force either a red $C_3$ or blue $S_n$ within $3n$ rounds (he starts with selecting $2n$ edges of a star, then he forces
a blue $S_n$ on the set of vertices incident to red edges). Thus $\rr(C_3, S_n)\le 3n$, so $\rr(\Codd, S_n)\le 3n$ and 
we see another example of online size Ramsey numbers which are much smaller 
than their classic counterpart. Let us mention that the constant 3 in the upper bound on  $\rr(C_3, S_n)$ is not far 
from the constant $2.6$ in the lower bound, implied by our main Theorem \ref{oddcon} below.  

In this paper we focus on games $\RR(\Codd,\Con_{n,m})$ and $\RR(C_3,P_n)$. 
Since $\rr(\cC,\Con_{n,m})\le\rr(\Codd,\Con_{n,m})$, in view of \eqref{cyccon} we have  
$\rr(\Codd,\Con_{n,m})\ge n+m-1$. We improve this lower bound.  Our strategy for Painter is based on a potential function method and the golden number $\varphi=(\sqrt5+1)/2$ plays a role in it. Here is the first main result of our paper. 

\begin{theorem}\label{oddcon}
For every $m,n\in\N$ such that $n-1\le m\le \binom n2$ 
$$\rr(\Codd,\Con_{n,m})\ge \varphi n + m-2\varphi+1.$$
\end{theorem}

We will prove this theorem in Section \ref{lower}.
As a side effect, we receive the lower bound $\rr(C_{2k+1},T_n)\ge 2.6n-3$ for $k\ge 1$ and  every tree $T_n$ on $n\ge 3$ vertices.

In order to find an upper bound for $\rr(\Codd,\Con_{n,m})$, we begin from estimating $\rr(C_3,P_n)$. Here is our second main result,
which will be proved in Section \ref{upper}.

\begin{theorem}\label{c3pn}
For $n\ge 3$
$$\rr(C_3,P_n)\le 3n-4.$$
\end{theorem}

Clearly $\rr(\Codd,\Con_{n,n-1})\le \rr(C_3,P_n)$ so Theorems \ref{oddcon} and \ref{c3pn} give the following bounds on $\rr(C_3,P_n)$.

\begin{cor}\label{cor1}
If $n\ge 3$, then
$$\big\lceil (\varphi+1)n -2\varphi\big\rceil\le \rr(C_3,P_n)\le 3n-4.$$
\end{cor}

It improves the best known multiplicative constants in the lower and upper bounds $2n-1\le \rr(C_3, P_n)\le 4n-5$, proved by  
Dybizba\'nski, Dzido and Zakrzewska \cite{dzido}. The upper and the lower bounds in Corollary \ref{cor1} are optimal for $n=3, 4$ 
since  $\rr(C_3, P_3)=5$ and $\rr(C_3, P_4)=8$, as it was verified (by computer) by  Gordinowicz and Pra\l at \cite{pral}.
We believe that the upper bound is sharp for every $n\ge 3$.

\begin{conj}
$\rr(C_3,P_n)=3n-4$ for every $n\ge 3$.
\end{conj}

In view of Theorem \ref{c3pn}, we have $\rr(\Codd,\Con_{n,n-1})\le 3n-4$. 
In the last section we will prove the following upper bound on $\rr(\Codd,\Con_{n,m})$.

\begin{theorem}\label{conupper}
For $n\ge 3$ and $n-1\le m\le (n-1)^2/4$
$$\rr(\Codd,\Con_{n,m})\le n+2m+O(\sqrt{m-n+1}).$$
\end{theorem}

\section{Preliminaries}

For a subgraph of $G$ induced on $V'\subs V(G)$, we denote the set of its edges by $E[V']$. If $V'\subs V(G)$ is empty, then
we define $E[V']=\emptyset$.

We say that a graph $H$ is coloured if every its edge is blue or red. A graph is red (or blue) if all its edges are red (blue). 
Let $G$ be a coloured graph. We say $G$ is red-bipartite if there exists a partition $V(G)=V_1\cup V_2$ (one of the two sets may be empty) such that there are no blue edges of $G$ between $V_1$ and $V_2$ and there are no red edges in $E[V_1]\cup E[V_2]$.
A pair of such sets $(V_1, V_2)$ is called a red-bipartition of $G$.
It is not hard to observe that a coloured graph $G$ is red-bipartite if and only if $G$ has no cycle with an odd number of red edges. 
Furthermore, every component of a coloured graph has at most one red-bipartition up to the order in the pair $(V_1, V_2)$.

\section{Proof of Theorem \ref{oddcon}}\label{lower}

Let $n,m\in\N$. Consider the following auxiliary game $\cG(n,m)$.  In every round Builder chooses a previously unselected 
edge from $K_\N$ and Painter colours it red or blue. 
The game ends if after a move of Painter there is a coloured cycle with an odd number of red edges or 
there exists a coloured connected graph $H$ with a red-bipartition $(V_1, V_2)$ such that for some $i\in \{1,2\}$ we have $|V_i|\ge n$ 
and $|E[V_i]|\ge m$. Builder tries to finish the game as soon as possible, while the goal of Painter is the opposite.
Let $\rr_{\cG(n,m)}$ be the minimum number of rounds in the game $\cG(n,m)$ provided both players play optimally. 

Clearly if there is a red odd cycle or a connected blue graph with $n$ vertices and $m$ edges on the board, then the game $\cG(n,m)$ ends so
$$\rr_{\cG(n,m)}\le \rr(\Codd,\Con_{n,m}).$$ 
Therefore in order to prove Theorem \ref{oddcon} it is enough to prove that $\rr_{\cG(n,m)}\ge \varphi n + m-2\varphi+1$.

We will define a strategy of Painter in $\cG(n,m)$ based on a potential function and prove that the potential function does not grow too much 
during the game. Let us define a function $f$ on the family of all coloured red-bipartite subgraphs of $K_{\N}$. 

Suppose $G=(V, E)$ is a coloured red-bipartite graph.  
If $G$ is an isolated vertex, then let $f(G)=0$. If $G$ is connected, with the red-bipartition $(V_1,V_2)$ and $|V|>1$, let 
\begin{eqnarray*}
p_G(V_i)&=&|V_i|\varphi+|E[V_i]|,\text{ for } i=1,2;\\
a(G)&=&\max(p_G(V_1),p_G(V_2)),\\ 
b(G)&=&\min(p_G(V_1),p_G(V_2)),\\
f(G)&=&\varphi a(G)-\varphi+ \max(a(G)-\varphi^3,b(G)).
\end{eqnarray*}
Finally, if $G$ consists of components $G_1, G_2,\ldots,  G_t$, then we put
$$f(G)=\sum_{i=1}^t f(G_i).$$

The motivation behind the choice of the potential function $f$ is described in Appendix \ref{GAD}.

We will use also the following function $g:\R^2\to\R$
$$g(x,y)=\varphi \max(x,y)-\varphi+ \max(\max(x,y)-\varphi^3,\min(x,y)).$$
Note that $g$ is symmetric, nondecreasing with respect to $x$ (and $y$) and if $(V_1,V_2)$ is a red-bipartition of a connected, coloured graph $H$, then
$g(p_H(V_1),p_H(V_2))=f(H)$. We can also rewrite $g$ as
$$g(x,y)=x+y-
\varphi+(\varphi-1)\max(x,y)+\max(x-y-\varphi^3,y-x-\varphi^3,0),$$
so $g$ is convex.

We are ready to present the strategy of Painter in $\cG(n,m)$. She will play so that after every round the coloured graph $G=(V(K_{\N}), E)$ 
containing all coloured edges is red-bipartite. Furthermore, we will show that it is possible to colour edges selected by Builder so that 
after every round the potential $f(G)$ increases by not more than $\varphi+1$.   The inductive argument is stated in the following lemma.

\begin{lemma}\label{gd3p}
Let $G$ be a coloured subgraph of $K_\N$ with $V(G)=V(K_\N)$ and the edge set consisting of all edges coloured within  $t\ge 0$ rounds of the game $\cG(n,m)$. Suppose that $G$ is red-bipartite and $e$ is an edge selected by Builder in round $t+1$. 
Then Painter can colour $e$ red or blue so that the obtained coloured graph $G'$ with $|E(G)|+1$ edges is red-bipartite and
$f(G')\le f(G)+\varphi+1$.
\end{lemma}

\begin{proof}

Let $G$ satisfy the assumptions of the lemma, let $e=uu'$ be an edge selected in round $t+1$, and suppose that $H,H'$ are the components of $G$ such that $u\in H$, $u'\in H'$, and $(V_1, V_2)$, $(V_1', V_2')$ are red-bipartitions of $H$ and $H'$, respectively.

We consider several cases. Below by $F+e_{red}$ and $F+e_{blue}$ we denote the coloured graphs obtain by adding $e$ to a coloured graph $F$ provided $e$ is coloured red or blue, respectively.   
\begin{enumerate}

\item 
$u,u'$ are in the same connected component of $G$, i.e. $H=H'$. 

If $u$ and $u'$ are in different parts of the red-bipartition, then Painter colours $e$ red. Then 
$f(H+e_{red})=f(H)$ and $f(G+e_{red})=f(G)$.

If $u$ and $u'$ are in the same part of the red-bipartition, say $u,u'\in V_1$, then Painter colours $e$ blue. Then we obtain the component $H''=H+e_{blue}$ with
$p_{H''}(V_1)=p_H(V_1)+1$ and $p_{H''}(V_2)=p_H(V_2)$. Therefore 
$f(H'')\le f(H)+\varphi+1$ and $f(G+e_{blue})\le f(G)+\varphi+1$.

\item 
$u,u'$ are isolated vertices in $G$. 

Then Painter colours $e$ red. The obtained graph  $G+e_{red}$ is red-bipartite and 
$$f(G+e_{red})=f(G)+\varphi \varphi-\varphi+\varphi=f(G)+\varphi^2=f(G)+\varphi+1.$$

\item
$u'$ is isolated in $G$, but $u$ is not. We may assume that $u\in V_1$.

\begin{enumerate}
\item[a.]  
Suppose that $p_H(V_2)>p_H(V_1)+\varphi+1$.
Then Painter colours $e$ blue. The obtained graph  $G+e_{blue}$ is red-bipartite and for $F=H+e_{blue}$ we have
$$a(F)=\max(p_{F}(V_1\cup\{u'\}),p_{F}(V_2))=\max(p_H(V_1)+\varphi+1,p_H(V_2))=p_H(V_2)=a(H)$$ 
and
$$b(F)=p_{F}(V_1\cup \{u'\})=b(H)+\varphi+1.$$ 
Therefore $f(G+e_{blue})\le f(G)+\varphi+1$.

\item[b.]  
Suppose that $p_H(V_2)\le p_H(V_1)+\varphi+1$.
Then Painter colours $e$ red. Thus $u'$ is added to $V_2$ in the red-bipartition of $F=H+e_{red}$ and
$$p_{F}(V_2+u')=p_H(V_2)+\varphi.$$

In order to estimate $f(F)$, we calculate $dg/dy(x_1,\cdot)$ for a fixed $x_1\in\R$. By definition of $g$ we have
$$
\frac{dg}{dy}(x_1,y)=\begin{cases}
0& \text{ for } y<x_1-\varphi^3,\\
1& \text{ for } x_1-\varphi^3< y<x_1,\\
\varphi& \text{ for } x_1< y<x_1+\varphi^3,\\
\varphi^2& \text{ for } y> x_1+\varphi^3.\\
\end{cases}
$$
Thus, for every $k\in(0,\varphi)$, in view of the fact that $p_H(V_2)+k\le p_H(V_1)+\varphi+1+k<p_H(V_1)+\varphi^3$, we have
$$\frac{dg}{dy}(p_H(V_1),p_H(V_2)+k)\le \varphi.$$ 
Therefore
\begin{eqnarray*}
f(F)-f(H)&=& g(p_{F}(V_1),p_{F}(V_2\cup \{u'\}))-g(p_H(V_1),p_H(V_2))\\
&=& g(p_{H}(V_1),p_{H}(V_2)+\varphi)-g(p_H(V_1),p_H(V_2))\le \varphi^2=\varphi+1.
\end{eqnarray*}

\end{enumerate}

\item \label{nieizol}
$u$ and $u'$ are not isolated in $G$ and $H\neq H'$. 

We may assume that $u\in V_1$ and $u'\in V_1'$. 
Let $a=\max(p_H(V_1),p_H(V_2))$, $b=\min(p_H(V_1),p_H(V_2))$, $c=a-b$,
$a'=\max(p_{H'}(V_1'),p_{H'}(V_2'))$, $b'=\min(p_{H'}(V_1'),p(V_2'))$ and $c'=a'-b'$.

\begin{enumerate}
\item[a.]
Suppose that $p_H(V_1)>p_H(V_2)+1$ and $p_{H'}(V_1')+1<p_{H'}(V_2')$.
Then Painter colours $e$ blue. Thus the components $H,H'$ of $G$ are joined into a component $F$ of $G+e_{blue}$, with the
red-bipartition $(V_1\cup V_1', V_2\cup V_2')$ of $F$. 
Notice that $p_H(V_1)=a$, $p_H(V_2)=b$,  $p_{H'}(V_1')=b'$, $p_{H'}(V_2')=a'$ and $c,c'>1$. Furthermore  
\begin{eqnarray*}
p_F(V_1\cup V_1')&=& p_H(V_1)+p_{H'}(V_1')+1=a+b'+1,\\
p_F(V_2\cup V_2')&=& p_H(V_2)+p_{H'}(V_2')=a'+b.
\end{eqnarray*}
Assume to the contrary that $f(G+e_{blue})>f(G)+\varphi+1$. This means that
$$f(H)+f(H')+\varphi+1< f(F),$$
or equivalently    
$$g(a,b)+g(a',b')+\varphi+1< g(a+b'+1,a'+b).$$
The above inequality and a general property of $g$ that $g(x+z,y+z)=g(x,y)+(\varphi+1)z$, imply that 
$$g(c,0)+g(c',0)+\varphi^2< g(c+1,c').$$
If $c<c'$, then we can swap $c$ and $c'$ and the left-hand side will not change while the right-hand side will get bigger (by the fact that $g$ is convex and symmetric). It means that we may assume that $c\ge c'$. Thus in view of the definition of $g$ we get
\begin{eqnarray*}
c\varphi+\max(c-\varphi^3,0)+c'\varphi+\max(c'-\varphi^3,0)-2\varphi+\varphi^2\\
<(c+1)\varphi-\varphi+\max(c+1-\varphi^3,c'),
\end{eqnarray*}
and hence 
$$\max(c-\varphi^3,0)+(c'-1)\varphi+\max(c'-\varphi^3,0)<\max(c-\varphi^3,c'-1).$$
This inequality implies that $\max(c-\varphi^3,(c'-1)\varphi)<\max(c-\varphi^3,c'-1)$, which for $c'>1$ leads to a contradiction.
Thereby we proved that $f(G+e_{blue})\le f(G)+\varphi+1$.

\item[b.]
Suppose that $p_H(V_1)>p_H(V_2)$ and $p_{H'}(V_1')<p_{H'}(V_2')$ but case \ref{nieizol}a does not hold, i.e. 
$p_H(V_1)\le p_H(V_2)+1$ or $p_{H'}(V_1')+1\ge p_{H'}(V_2')$. 
Then Painter colours $e$ red. Similarly to the previous case, the components $H,H'$ of $G$ are joined into a component $F$ of $G+e_{red}$ but the red-bipartition of $F$ is $(V_1\cup V_2', V_2\cup V_1')$ and 
$$
p_F(V_1\cup V_2')=a+a',\
p_F(V_2\cup V_1')= b+b'.
$$

Again, assume the contrary that $f(G+e_{red})>f(G)+\varphi+1$, which is equivalent to 
$f(H)+f(H')+\varphi+1< f(F)$ and hence
$$g(a,b)+g(a',b')+\varphi^2< g(a+a',b+b').$$
Then we get
$$g(c,0)+g(c',0)+\varphi+1<g(c+c',0)$$
and in view of the definition of $g$
$$\max(c-\varphi^3,0)+\max(c'-\varphi^3,0)+1<\max(c+c'-\varphi^3,0).$$
By the assumptions of this case we have $c\le 1$ or $c'\le 1$ so the above inequality cannot hold.
\end{enumerate}

Because of the symmetric role of $H$ and $H'$, cases \ref{nieizol}a and   \ref{nieizol}b cover all situations with
$(p_H(V_1)-p_H(V_2))(p_{H'}(V_1')-p_{H'}(V_2'))<0$. It remains to analyse the opposite case.

\begin{enumerate}
\item[c.] Assume that $(p(V_1)-p(V_2))(p(V_1')-p(V_2'))\ge 0$. 
Then Painter colours $e$ red. As in the previous case, the components $H,H'$ of $G$ are joined into a component $F$ of $G+e_{red}$  with its red-bipartition $(V_1\cup V_2', V_2\cup V_1')$ but we have either
$$
p_F(V_1\cup V_2')=a+b'\text{ and }
p_F(V_2\cup V_1')= b+a',
$$
or     
$$
p_F(V_1\cup V_2')=b+a'\text{ and }
p_F(V_2\cup V_1')=a+b'.
$$
In both cases, by the symmetry of the function $g$, the assumption that $f(H)+f(H')+\varphi+1< f(F)$ leads to 
$$g(a,b)+g(a',b')+\varphi+1<g(a+b',b+a').$$
Then
$$g(c,0)+g(c',0)+\varphi+1<g(c,c').$$
The inequality is symmetric with respect to $c$ and $c'$ so we may assume $c\ge c'$. Then the above inequality is equivalent to
$$\max(c-\varphi^3,0)+\max(c'-\varphi^3,0)+c'\varphi+1<\max(c-\varphi^3,c').$$
It implies that   $\max(c-\varphi^3,c'\varphi)<\max(c-\varphi^3,c')$  and again we get a contradiction.
\end{enumerate}

\end{enumerate}

\end{proof}

We have proved that Painter has a strategy in $\cG(n,m)$ such that after every round of the game $\cG(n,m)$ 
the graph induced by the set of all coloured edges is red-bipartite and in every round its potential $f$ increases by not more than $\varphi+1$
(at the start of the game the potential is 0). We infer that a coloured cycle with an odd number of red edges never appears in the game. 
Suppose that (given the described strategy of Painter and any strategy of Builder) after $t$ rounds of the game 
the graph $G$ induced by the set of all coloured edges contains a component $H$ with a red bipartition $(V_1, V_2)$ such that 
$|V_1|\ge n$ and $|E[V_1]|\ge m$.  On the one hand, we have $f(G)\le t(\varphi+1)$. On the other hand, we have
$$f(G)\ge f(H)=\varphi a(H)-\varphi+ \max(a(H)-\varphi^3,b(H))\ge (\varphi+1)a(H)-\varphi-\varphi^3
\ge (\varphi+1)(\varphi n+m)-\varphi-\varphi^3.$$
Therefore 
$$t\ge \frac{(\varphi+1)(\varphi n+m)-\varphi-\varphi^3}{\varphi+1}=\varphi n+m-2\varphi+1.$$
We conclude that Painter can survive at least $\varphi n+m-2\varphi+1$ in the game $\cG(n,m)$. In view of previous remarks, 
it proves also that $\rr(\Codd,\Con_{n,m})\ge \varphi n+m-2\varphi+1$.

\section{Proof of Theorem \ref{c3pn}}\label{upper}

If $n=3$, then we know that $\rr(C_3,P_n)=5=3n-4$. One can also check that Builder can force a red $C_3$ or a blue $P_3$
playing on $K_5$ only -- we will use this fact later.  

Throughout this section, we assume that $n\ge 4$. While considering a moment of the game $\RR(C_3, P_n)$ we say 
that Builder connects two paths $P$ and $P'$ if in the considered round he selects an edge incident to an end of $P$ and an end of $P'$. We say that Builder can force a blue graph $H$ within $t$ rounds if Builder has a strategy such that after at most $t$ next rounds a red copy of $C_3$ or a blue copy of $H$ is created. In other words, we assume that Painter avoids moves that lose immediately (if possible). 

Let $P(s,t)$ be a coloured path on $s+t$ vertices obtained from two vertex-disjoint blue paths $P_s$ and $P_t$ by connecting their 
ends with a red edge. Moreover, let $P(s,0)$ denote a blue path on $s$ vertices.
A maximal (in sense of inclusion) coloured path is called a $brb$-path if it is isomorphic to $P(s,t)$ for some $s,t>0$. 
We say that a blue path is pure if it is maximal (in sense of inclusion) and none of its vertices is incident to a red edge.

Let $u(P(s,t))=|s-t|$ for every $s,t\ge 0$. 
Thus $u$ is a measure of the path ``imbalance''. We say that $P(s,t)$ is balanced if $u(P(s,t))=0$ and imbalanced otherwise. 

We start with the following observation. 

\begin{lemma} \label{phase2}
Suppose that $a_i\ge b_i$ for $i=0,1,\ldots,k$ and a coloured graph $G$ contains vertex-disjoint paths $P(a_0,b_0),P(a_1,b_1),\ldots,P(a_k,b_k)$. 
Then Builder can force either a red $C_3$ or a blue path on $a_0+\sum_{i=1}^k b_k$ vertices within at most $2k$ rounds.
\end{lemma}
\begin{proof}
We prove it by induction. The case $k=0$ is trivial. If $k>0$, then Builder can connect the end of the blue path $P_{a_0}$ 
with both ends of a red edge in $P(a_k,b_k)$ and force a blue path of length at least $a_0+b_k$. 
Now we have vertex-disjoint paths $P(a_0+b_k,0),P(a_1,b_1),\ldots,P(a_{k-1},b_{k-1})$, so by the induction hypothesis 
Builder can force a blue path of length $a_0+b_k+\sum_{i=1}^{k-1} b_i$ in the next $2k-2$ moves.
\end{proof}

Consider the following strategy for Builder in $\RR(C_3, P_n)$. 
Builder will assume that the board of the game is $K_{2n-1}$, other edges of $K_\N$ are ignored. 
We divide the game into two stages. Roughly speaking, in Stage 1 Builder plays so that many $brb$-paths are created and  
one of the coloured paths is more imbalanced than all the others put together. In Stage 2, Builder applies his strategy from 
Lemma \ref{phase2} and forces a long blue path.

More precisely, at the beginning of the game, we have $2n-1$ pure blue paths (they are trivial).  
In every round of Stage 1, Builder connects two shortest pure blue paths.  
Painter colours the selected edge red or blue. 
Builder continues doing so as long as there are at least two pure blue paths on the board, with one exception: 
If Painter colours the first $n-3$ edges red, then Stage 1 ends.

If this exception happens,  the coloured graph consists of $n-3$ isolated red edges and five isolated vertices $u_1,\ldots,u_5$. 
The game proceeds to Stage 2. 
Builder forces a blue $P_3$ within five next rounds, using only edges of the graph $K_5$ on the vertex set $\{u_1,\ldots,u_5\}$. 
After that the coloured graph contains a path $P(3,0)$ and $n-3$ copies of $P(1,1)$, 
so in view of Lemma \ref{phase2} Builder can force a blue path on $3+n-3=n$
vertices or a red $C_3$ within next $2(n-3)$ rounds. Thus in this case the game ends after at most $n-3+5+2(n-3)=3n-4$ rounds.

Further we assume that Painter colours blue at least one of the edges selected in the first $n-3$ rounds. 
Stage 1 ends when there is at most one pure blue path on the board. Observe that at the end of Stage 1
the graph induced by all coloured edges consists of vertex-disjoint $brb$-paths and at most one pure blue path.

Since Builder always connects two shortest pure blue paths, we infer that the following holds. 

\begin{fact} \label{times2}
After every round of Stage 1, if $P_k$, $P_l$ are pure blue paths, then $k\le 2l$.
\end{fact}

Let us verify that also the following is true.

\begin{fact} \label{most}
After every round of Stage 1, there is at most one pure blue path in which the number of vertices is not a power of $2$.
\end{fact}

Indeed, suppose that this property holds after some round $t$, in the next round Builder connects  pure blue paths are $P_{a}$ and $P_{b}$,
and after round $t+1$, the number of pure blue paths in which the number of vertices is not a power of $2$ increases.
It is possible only if the edge in round $t+1$ was coloured blue, $a=2^k$ and $b=2^l$, with some integers $k\neq l$. 
From Fact \ref{times2} we know that it is only possible when $k+1=l$, $P_{2^k}$ was the only shortest pure blue path and $P_{2^l}$ was 
both the second shortest and the longest pure blue path after round $t$. It means that after round $t+1$ there is a pure blue 
path $P_{a+b}$ and every other pure blue path has $2^l$ vertices. 

The next lemma gives some insight into the imbalance properties of the collection of the coloured paths in Stage 1. 

\begin{lemma}\label{imb}
Let  $H_1,H_2,\ldots,H_q$ be the sequence of all imbalanced $brb$-paths at the moment of the game and suppose 
they appeared in that order (i.e.~$H_1$ was created first in the game, $H_2$ was second and so on).
Then $u(H_{k+1})\ge 2u(H_k)$ for $k=1,2, \ldots, q-1$. Moreover for any pure blue path $H$ we have $u(H)\ge 2u(H_q)$.
\end{lemma}

\begin{proof}
Let us present the inductive argument. Suppose the assertion is true for the sequence $H_1,H_2,\ldots,H_{q-1}$
and consider the round $r$ after which $H_{q-1}=P(s,t)$ appeared. From Facts \ref{times2} and \ref{most} we know that $s,t\in [2^k,2^{k+1}]$ for some integer $k$, so $u(H_{q-1})\le 2^k$. By the two facts and by the minimality of the blue paths generating $H_{q-1}$ we infer that at the end of round $r$ all pure blue paths have $2^{k+1}$ vertices. Therefore $2^{k+1}|u(H)$ for any pure blue path $H$ created after round $r$, until the end of Stage 1. Thus $2u(H_{q-1})\le u(H)$ for any such path $H$. Furthermore, since  $H_q$ is created after round $r$ by connecting two pure blue paths on, say, $s'$ and $t'$ vertices, we have $u(H_q)=|s'-t'|>0$ and this 
number is divisible by $2^{k+1}$. Therefore $2u(H_{q-1})\le u(H_q)$. The argument that $2u(H_{q})\le u(H)$ for every pure blue path
is analogous to the one for $H_{q-1}$. 
\end{proof}

Let us consider the position at the end of Stage 1. Let us recall that then we have at least one blue edge on the board $K_{2n-1}$, 
at most one pure blue path and a collection of $brb$-paths (if any). 
If there is a pure blue path, then let $H_0$ be that path. Otherwise let $H_0$ be the last imbalanced $brb$-path 
that appeared in Stage 1. Let $H_1, H_2,\ldots, H_m$ be the sequence of all imbalanced $brb$-paths which appeared (in that order) in Stage 1, except for the path $H_0$. It follows from Lemma \ref{imb} that 
\begin{equation}\label{uu}
\sum\limits_{j=1}^m u(H_j)\le \sum\limits_{j=0}^{m-1} \frac{u(H_m)}{2^j}\le 2 u(H_m)\le u(H_0).
\end{equation}
Let $H_1', H_2',\ldots, H_l'$ be the family of all $brb$-paths which are balanced (at the end of Stage 1).  

In order to calculate the number of rounds in Stage 1, notice that the subgraph $G$ of $K_{2n-1}$ with $V(G)=V(K_{2n-1})$ whose edge set 
consists of all edges coloured in Stage 1 is a union of $m+l+1$ vertex-disjoint paths so Stage 1 lasts $|E(G)|= |V(G)|-m-l-1=2n-2-m-l$ rounds. 

Observe also that $G$ has no isolated vertices. Indeed, if such a trivial blue path existed at the end of Stage 1, then
in every previous round, Builder would have connected two trivial blue paths and Painter has coloured the selected edges red. 
It would contradict the assumption that there is a blue edge in Stage 1. Thus $G$ has no isolated vertices and thus it has at most 
$n-1$ components. We conclude that 
\begin{equation}\label{comp}
m+l+1\le n-1.
\end{equation}

After Stage 1 the game proceeds to Stage 2. 

Note that for any path $P=P(s,t)$ (balanced or imbalanced)  with $s\ge t$ we have $s=(|V(P)|+u(P))/2$ and $t=(|V(P)|-u(P))/2$.
Let us also recall that for Builder the board of the game is $K_{2n-1}$ so 
$$\sum\limits_{j=0}^m |V(H_j)|+\sum\limits_{j=1}^l |V(H_j')|=2n-1.$$
In Stage 2, Builder applies his strategy from Lemma \ref{phase2} to the paths 
$H_0, H_1,\ldots,H_m, H_1',H_2',\ldots,H_l'$. Thereby he forces a blue path on $t$ vertices, where 
\begin{eqnarray*}
t&=&\frac12\Big(|V(H_0)|+u(H_0)+\sum\limits_{j=1}^m (|V(H_j)|-u(H_j))+\sum\limits_{j=1}^l (|V(H_j')|-u(H_j'))\Big)\\
&=&\frac12\Big(2n-1+u(H_0)-\sum\limits_{j=1}^m u(H_j)-\sum\limits_{j=1}^l u(H_j')\Big)\\
&=&n-\frac12+\frac 12\Big(u(H_0)-\sum\limits_{j=1}^m u(H_j)\Big)\ge n-\frac12.
\end{eqnarray*}
The last inequality follows from \eqref{uu}. Stage 2 ends. 
Thus in Stage 2, Painter forces either a red $C_3$ or a blue $P_n$ and Stage 2 lasts, in view of Lemma \ref{phase2}, at most $2(m+l)$ rounds.

We conclude that the number of rounds in Stage 1 and Stage 2 is not greater than $2n-2-m-l+2(m+l)=2n-2+m+l$. 
Because of \eqref{comp}, the game $\RR(C_3,P_n)$ lasts at most $2n-2+n-2=3n-4$ rounds. 

\section{Proof of Theorem \ref{conupper}}

Let $n\ge 3$ and $n-1\le m\le(n-1)^2/4\le \binom{\lfloor n/2\rfloor }{2}+\binom{\lceil n/2\rceil }{2}$.
In view of Theorem \ref{c3pn} it is enough to assume that $m\ge n$. 
Let $k$ be the smallest integer such that $1\le k\le n$ and $m\le n-k+\binom{\lfloor k/2\rfloor }{2}+\binom{\lceil k/2\rceil }{2}$.  
Builder begins the game $\RR(\Codd,\Con_{n,m})$ by forcing either a red triangle or a connected blue graph $n$ vertices. 
By Theorem \ref{c3pn} it takes him at most $3n-4$ rounds. 
Suppose that after there is no red cycle on the board and denote the blue graph on $n$ vertices by $B$. 

Let $B'$ be a connected blue subgraph of $B$ with $|V(B')|=k$. Clearly $|E(B')|\ge k-1$ and $|E(B)\setminus E(B')|\ge n-k$. 
Further in the game, Builder selects all edges of $E[V(B')]$ which are not coloured yet. It takes him at most 
$\binom{k}{2}-|E(B')|\le \binom{k}{2}-(k-1)$ rounds. If there is no red triangle in the resulting coloured complete graph 
on $k$ vertices, then by Tur\'an's theorem at least $\binom{\lfloor k/2\rfloor }{2}+\binom{\lceil k/2\rceil }{2}$ of its edges are blue. 
There are also at least $n-k$ blue edges in $E(B)\setminus E(B')$ so after at most 
$3n-4+\binom{k}{2}-(k-1)$ rounds 
we have a blue graph on $n$ vertices with at least 
$n-k+\binom{\lfloor k/2\rfloor }{2}+\binom{\lceil k/2\rceil }{2}\ge m$ edges.
Thus 
$$\rr(\Codd,\Con_{n,m})\le 3n+\binom{k}{2}-k-3.$$

It follows from the definition of $k$ that $m=n+\frac{k^2}{2}+O(k)$ and $k=O(\sqrt{m-n+1})$ and hence
$$\rr(\Codd,\Con_{n,m})\le n+2m+O(\sqrt{m-n+1}).$$

\bibliographystyle{amsplain}

\appendix
\section{Grants and Donations}\label{GAD}

In this section we present some intuition behind the choice of the potential function used in Section \ref{lower}.
We start with a game that at first glance is unrelated to $\RR(C_3, P_n)$. 

Given natural numbers $d,r,b$ and $N$, consider the following Grants and Donations game, which will be denoted by $GD(d,r,b, N)$. There are two players, Sponsor and (grant) Committee. There are also two teams of researchers (who are not players in this game). Every month Sponsor decides what he does with his money. He may either make a donation of $d$ dollars directly to one of the teams, or to suggests one team to Committee. In the latter case Committee either gives $b$ dollars to the suggested team and we will say that Committee gave a blue grant, or it gives $r$ dollars to the other team -- this will be called a red grant from the Committee. The game ends when one of the teams has at least $N$ dollars. 
The goal of Sponsor is to end the game as soon as possible and the goal of Committee is the opposite.
Let $\rr(GD(d,r,b, N))$ be the number of rounds in the game where both players play optimally.

If $r\le d$ or $b\le d$, then the best strategy for Sponsor is, obviously, to give the money directly to the first team, so let's assume that $d< r,b$. Furthermore, without loss of generality we can assume that $d<r\le b$.

Now suppose that Sponsor gave money to a team and in the next month he suggests a team to Committee. Sponsor can switch the order of these moves. In the latter case, Committee is less informed at the moment of making a move, so it cannot make a better decision for itself. This observation leads to the following fact.

\begin{fact}
Let $S$ be the set of optimal strategies for Sponsor. In $S$ there exists a strategy such that  Sponsor never suggests a team to Committee after making a donation to any team.
\end{fact}

Let us consider two Sponsor's strategies such that Sponsor never suggests a team to Committee after making a donation to any team. The first one is: always give money to the first team. Then the game ends after $\lceil N/d\rceil$ rounds. The second strategy is: always suggest the first team to Committee. Then Committee can make at most $\lceil N/b\rceil$ blue moves and at most $\lceil N/r\rceil$ red moves (and it cannot achieve both numbers). In this case the game ends after at most $\lceil N/r\rceil+\lceil N/b\rceil-1$ rounds. 

Now let us consider the following strategy of Committee. It gives $b$ dollars to the team chosen by Sponsor if and only if this team has got at least $b$ dollars less from Committee than the other team (we don't consider here the money from Sponsor). With this strategy, if the teams have got $k$ and $k'$ dollars from Committee, respectively, and $k\ge k'$, then at least $k$ dollars were given in red grants and $k-r-b<k'$. Therefore there were more than $k/r+(k-r-b)/b=k(1/r+1/b)-(r+b)/b$ grants, i.e. Committee's moves. If the game ended, i.e. one of the teams has $N$ dollars, then there was at least $(N-k)/d+k(1/r+1/b)-(r+b)/b\ge N\min (1/d,1/r+1/b)-(r+b)/b$ rounds. 

By the above analysis we come to the following conclusion.

\begin{theorem}\label{gd}
Let $d<r\le b$ and $M=N \min(1/d,1/r+1/b)$. Then
$$M-1-r/b<\rr(GD(d,r,b,N))<M+1.$$
\end{theorem}

Note that since $r\le b$, bounds of this theorem form an interval of length smaller or equal to 3.

Now we return to the game $\cG(n,m)$ defined in Section \ref{lower} and consider its Connector-Painter version.
The rules of the Connector-Painter version are the same as the rules of $\cG(n,m)$, with the additional requirement that the edges selected by Connector (Builder) have to induce a connected graph after every round. We denote by $\rr^{con}_{\cG(n,m)}$ the number of rounds of this game provided Connector and Painter play optimally. 
We are going to use Theorem \ref{gd} to find a lower bound for $\rr^{con}_{\cG(n,m)}$. 

\begin{cor}\label{conn}
Let $k>1$ be a real number, $m,n$ be integers, $m\ge 0$ and $n\ge 2$. Then 
$$\rr^{con}_{\cG(n,m)}>(k(n-1)+m)\min(1,1/k+1/(k+1))-k/(k+1).$$
In particular,
$$\rr^{con}_{\cG(n,m)}>\varphi n+m-2\varphi+1.$$
\end{cor}

\begin{proof}
In order to find a lower bound, we define a strategy for Painter. Let Painter colour the first edge red. In every next round she plays so that the graph induced by coloured edges is red-bipartite. Since it is also 
connected it has exactly one red-bipartition $(V_1, V_2)$. Observe that such a strategy is possible. Indeed, given the 
red-bipartition $(V_1, V_2)$ and a new edge $e$ selected by Connector, if $e$ has one end in $V_1$ and the other in $V_2$, then Painter colours it red; if $e$ has one end in $V_1$ and the other end is a new vertex, then Painter can either colour $e$ red and the new vertex is added to $V_2$, or she can colour $e$ blue -- then the new vertex is added to $V_1$ and a blue edge is added to $E[V_1]$.

The game ends exactly when $|V_i|\ge n$ and there are at least $m$ blue edges in $E[V_i]$ for $i=1$ or $i=2$.
Observe that this game is very similar to a Grants and Donations game: Sponsor, Committee and research groups correspond to Connector, Painter, $V_1$ and $V_2$, respectively. The $i$-th research group has two types of ``money'': vertices and blue edges in $V_i$. Nevertheless, we can assign the value 1 to an edge and $k$ to a vertex so that the funds are in dollars. After her first move in the Connector-Painter game, Painter can pretend that she is playing $GD(1,k,k+1, N)$, where $N=(n-1)k+m$ (note that after the first move $|V_1|=|V_2|=1$ and there is no blue edge in $E[V_1], E[V_2]$). Therefore
$$\rr^{con}_{\cG(n,m)}\ge 1+\rr(GD(1,k,k+1,(n-1)k+m)>(k(n-1)+m)\min(1,1/k+1/(k+1))-k/(k+1).$$
For $k=\varphi$ we get the second part of the thesis.
\end{proof}

The above result gives some insight into the game $\cG(n,m)$ if Builder keeps the coloured graph connected. 
In the general case, one has to overcome technical difficulties arising from many components of the coloured graph and therefore the potential function defined in Section \ref{lower} is more complicated. 


\end{document}